\documentclass[leqno]{article}
\usepackage{amsfonts}
\usepackage{amssymb}
\usepackage{amsmath}
\usepackage{amsthm}
\usepackage{xr}
\usepackage{graphicx}
\usepackage{color}
\usepackage{hyperref}

\providecommand{\U}[1]{\protect\rule{.1in}{.1in}}
\newtheorem{theorem}{Theorem}[section]

\newtheorem{proposition}[theorem]{Proposition}

\def\limfunc#1{\mathop{\rm #1}}
\def\func#1{\mathop{\rm #1}\nolimits}

\begin{document}

\title{Maximum principle for an optimal control problem associated to a SPDE
with nonlinear boundary conditions}
\author{Stefano Bonaccorsi$^{a}$,\quad Adrian Z\u{a}linescu$^{a,b}$ \\
\bigskip\\
{\small $^{a}$~Dipartimento di Matematica, Universit\`{a} di Trento, via
Sommarive 14, 38123 Povo (Trento), Italy.}\\
{\small $^{b}$~\textquotedblleft O. Mayer\textquotedblright\ Mathematics
Institute of the Romanian Academy, Ia\c{s}i,}\\
{\small Carol I Blvd., no. 8, Ia\c{s}i, 700506, Romania.}}
\date{}
\maketitle

\begin{abstract}
We study a control problem where the state equation is a nonlinear partial
differential equation of the calculus of variation in a bounded domain,
perturbed by noise. We allow the control to act on the boundary and set
boundary conditions which result in a stochastic differential equation for
the trace of the solution on the boundary. This work provides necessary and
sufficient conditions of optimality in the form of a maximum principle. We
also provide a result of existence for the optimal control in the case where
the control acts linearly.\newline
\textbf{Keywords:} stochastic control, maximum principle, stochastic
evolution equation, backward stochastic differential equation\newline

\textit{MSC 2010}: 93E20, 60H15, 60H30
\end{abstract}

\setcounter{footnote}{1}

\section{Introduction}

Let $\mathcal{O}\subseteq \mathbb{R}^{n}$ be a bounded smooth domain with
regular boundary $\Gamma :=\partial \mathcal{O}$ and outward normal unit
vector $\nu $; we also fix a terminal time $T>0$. We fix a nonlinear
operator $\func{div}\mathbf{a}(x,\nabla y)$ of Leray--Lions type, and we
consider the following controlled nonlinear diffusion equation with
dynamical boundary conditions:

\begin{equation}
\left\{
\begin{tabular}{ll}
$dy(t,x)=\func{div}\mathbf{a}(x,\nabla y)\,dt+b\,dw(t,x),$ & $\left(
t,x\right) \in\left( 0,T\right) \times\mathcal{O};\medskip$ \\
$dy(t,\xi)=\left[ -\mathbf{a}(\xi,\nabla y)\cdot\nu-\gamma(\xi,y(t,\xi
),u(t,\xi))\right] \,dt+\tilde{b}\,d\tilde{w}(t,\xi),$ & $\left(
t,\xi\right) \in\left( 0,T\right) \times\Gamma;\medskip$ \\
$y(0,x)=y_{0}(x).$ & $x\in\mathcal{\bar{O}}.$%
\end{tabular}
\ \ \ \ \ \ \ \ \ \ \ \ \ \ \ \right.   \label{Eq_state0}
\end{equation}
$w$ and $\tilde{w}$ are independent infinite dimensional Wiener processes
with values in $L^{2}(\mathcal{O})$ and $L^{2}(\Gamma)$, respectively. We
assume that $u$ is an admissible control acting on the boundary and we study
the problem of minimizing the cost functional%
\begin{align}
J(u):= & \mathbb{E}\left\{ \int_{0}^{T}\left[ \int_{\mathcal{O}%
}\ell(x,y(s,x))\,dx+\int_{\Gamma}\bar{\ell}(\xi,y(s,\xi),u(s,\xi))\,d\xi %
\right] \,ds\right\}  \label{Cost_functional0} \\
& +\mathbb{E}\left[ \int_{\mathcal{O}}\psi(x,y(T,x))dx+\int_{\Gamma}\bar{\psi%
}(\xi,y(T,\xi))\,d\xi\right]  \notag
\end{align}

Equations of the form (\ref{Eq_state0}), called fully parabolic boundary
value problem in the seminal paper of Escher \cite{Esc93}, have been
considered also in the stochastic setting, see \textit{e.g.}\ Chueshov and
Schmalfu\ss\ \cite{ChuSch04}, Bonaccorsi and Ziglio \cite{BonZig14} and
Barbu, Bonaccorsi and Tubaro \cite{BarBonTub15}. Such problems are used to
describe a wide variety of physical processes, among which we mention heat
propagation in a plasma gas, population dynamics and other nonlinear
diffusive phenomena (e.g., see \cite{DaPZab92}). It should be noticed that
boundary conditions of the form prescribed in \eqref{Eq_state0} are of a
non-standard type; nevertheless, dynamical boundary conditions, \emph{i.e.}
involving formally a time derivative of the solution on the boundary are
used as a model in several physical systems
, see the paper \cite{Goldstein2006} for a derivation and a physical
interpretation in the case of the heat equation; further applications are
given to heat transfer in a solid imbedded in a moving fluid \cite[\S 7.2]%
{Vra03}, surface gravity waves in oceanic models (\cite{DiaTel99}, \cite%
{DuaGaoSch02}, \cite{Mul97}), as well as in fluid dynamics \cite{Su93},
phase separation phenomena \cite{DuaYan07}, and this list is far from being
exhaustive.

In our setting, existence for the solution of equation (\ref{Eq_state0}) is
proven in \cite{BonZig14} or \cite{BarBonTub15} via an operatorial approach
which allows to rewrite the system as a stochastic differential equation in
the product space $H^{1}(\mathcal{O})\times L^{2}(\Gamma)$. A similar
approach was recently developed for a class of deterministic parabolic
equation with Wentzell boundary conditions in \cite{BarFavMar15}.

Our objective is to control such a system through the boundary, considering
that in practice it is easier to implement boundary control than distributed
parameter controls (see \cite{DebFuhTes07} for a discussion about the
subject). Such control problems have been widely studied in the
deterministic literature (see \cite{LasTri91}) and have been addressed in
the stochastic case as well (see \cite{DunMasPas98}, \cite{GozRouSwi00},
\cite{Ich85}, \cite{Mas95}, \cite{DebFuhTes07}). With regard to dynamical
boundary conditions, we mention that an associated control problem have
already been addressed by Bonaccorsi, Confortola, Mastrogiacomo \cite%
{BonConMas08}, following the backward SDEs (BSDEs, for short) approach
introduced by Fuhrman and Tessitore in \cite{FuhTes02} in an abstract
setting. We emphasize that in general the above papers concern
one-dimensional domains.

The present article deals with the control problem from a different point of
view. We will follow the maximum principle approach, which has been
introduced by Pontryagin and his group in the 1950's in order to establish
necessary conditions of optimality for deterministic controlled systems.
Towards the extension to the stochastic controlled systems one difficulty is
that the adjoint equation becomes a linear BSDE, especially for stochastic
PDEs (SPDEs), in which case the respective BSDE can be seen as a backward
SPDE (BSPDE, for short). Several papers are devoted to the study of maximum
principles for SPDEs; see, e.g., \cite{Ben83}, \cite{HuPen90}, \cite%
{OksSulZha14}. Stochastic maximum principle for SPDEs with noise and control
on the boundary was established by Guatteri \cite{Gua11} and Guatteri and
Masiero \cite{GuaMas13}, in the case of an interval in $\mathbb{R}$. Their
treatment, based on semigroup theory, is different from ours; in this paper
we deal with variational solutions for the controlled system, as well as for
the adjoint equation.

The paper is organized as follows. In section 2, we introduce some notations
and recall some preliminary results concerning the well-posedness of the
state equation. Section 3 is devoted to the derivation of necessary and
sufficient optimality conditions in the form of a maximum principle. In
order to achieve this, we use the duality between the adjoint equation and
the variation equation. We will first analyze the adjoint equation, for
which we give an existence theorem based on a result of M\'{a}rquez-Dur\'{a}%
n and Real \cite{MarDur04} concerning BSDEs in a variational framework.
Then, the variation equation is obtained by using a linear perturbation of
the control. In section 4, we prove directly the existence of an optimal
control under the assumption that the coefficient $\gamma$ depends linearly
on the control.

\section{Preliminaries}

Let $\mathcal{O}\subseteq\mathbb{R}^{n}$ be a bounded domain which is
sufficiently regular (see, e.g. \cite{AdaFou03}, Remark 7.45 or \cite%
{DemDem07}). On $\mathcal{O}$ we introduce the standard Sobolev space $H^{1}(%
\mathcal{O})$; on the boundary $\Gamma:=\partial\mathcal{O}$ we consider the
fractional order Sobolev space%
\begin{equation*}
H^{\frac{1}{2}}(\Gamma):=\left\{ \bar{y}\in L^{2}(\Gamma)\mid\int_{\Gamma
}\int_{\Gamma}\frac{\left\vert \bar{y}(\xi)-\bar{y}(\xi^{\prime})\right\vert
^{2}}{\left\vert \xi-\xi^{\prime}\right\vert ^{n}}d\xi
d\xi^{\prime}<+\infty\right\} ,
\end{equation*}
endowed with the norm%
\begin{equation*}
\left\Vert \bar{y}\right\Vert _{H^{\frac{1}{2}}(\Gamma)}^{2}:=\left\Vert
\bar{y}\right\Vert _{L^{2}(\Gamma)}^{2}+\int_{\Gamma}\int_{\Gamma}\frac{%
\left\vert \bar{y}(\xi)-\bar{y}(\xi^{\prime})\right\vert ^{2}}{\left\vert
\xi-\xi^{\prime}\right\vert ^{n}}d\xi d\xi^{\prime},\ \bar{y}\in H^{\frac{1}{%
2}}(\Gamma).
\end{equation*}
The following result of compactness\footnote{%
an operator is \textit{compact} if it maps bounded sets into precompact sets}
of the injection will be useful later:

\begin{equation*}
H^{\frac{1}{2}}(\Gamma)\hookrightarrow L^{2}(\Gamma),\ \text{compactly.}
\end{equation*}
It is well-known that for a smooth domain $\mathcal{O}$, the \emph{trace}
operator $\tau:H^{1}(\mathcal{O})\rightarrow L^{2}(\Gamma)$, with the
property that $\tau(y)=y%
\vert\big.%
_{\Gamma}$, $\forall y\in H^{1}(\mathcal{O})\cap C(\mathcal{\bar{O}})$, is
well-defined. Moreover, the range of $\tau$ is actually $H^{\frac{1}{2}%
}(\Gamma)$ and%
\begin{equation*}
\left\Vert \tau(y)\right\Vert _{H^{\frac{1}{2}}(\Gamma)}\leq K\left\Vert
y\right\Vert _{H^{1}(\mathcal{O})},\ \forall y\in H^{1}(\mathcal{O})
\end{equation*}
for some constant $K$ depending only on $\mathcal{O}$.

In what follows we suppose that the domain $\mathcal{O}$ is bounded and
smooth. We introduce the \textquotedblleft pivot\textquotedblright\ space $%
H:=L^{2}(\mathcal{O})\times L^{2}(\Gamma)$ endowed with the natural inner
product%
\begin{equation*}
\left\langle (y,\bar{y}),(y^{\prime},\bar{y}^{\prime})\right\rangle
_{H}:=\left\langle y,y^{\prime}\right\rangle _{L^{2}(\mathcal{O}%
)}+\left\langle \bar{y},\bar{y}^{\prime}\right\rangle _{L^{2}(\Gamma )},\ (y,%
\bar{y}),(y^{\prime},\bar{y}^{\prime})\in H
\end{equation*}
and norm $\left\Vert \cdot\right\Vert _{H}$. Let us consider the Banach space%
\begin{equation*}
V:=\left\{ (y,\bar{y})\in H^{1}(\mathcal{O})\times H^{\frac{1}{2}%
}(\Gamma)\mid\bar{y}=\tau(y)\right\} ;
\end{equation*}
endowed with the norm%
\begin{equation*}
\left\Vert (y,\bar{y})\right\Vert _{V}:=\left\Vert \nabla y\right\Vert
_{L^{2}(\mathcal{O})}+\left\Vert \bar{y}\right\Vert _{L^{2}(\Gamma)}.
\end{equation*}
The embedding $V\hookrightarrow H$ is compact; this property will be used in
the proof of Theorem \ref{Th: Exist}. Furthermore, the space $V$ is
isomorphic to $H^{1}(\mathcal{O})$ and it is densely embedded in $H$. Let $%
V^{\ast}$ be the dual space of $V$, with the dualization denoted $%
_{V^{\ast}}\left\langle \cdot,\cdot\right\rangle _{V}$. We fix the Gelfand
triple $V\subseteq H\subseteq V^{\ast}$ (the last formal inclusion implies
that $_{V^{\ast}}\left\langle \mathbf{z},\mathbf{y}\right\rangle
_{V}=\left\langle \mathbf{z},\mathbf{y}\right\rangle _{H}$ for every $%
\mathbf{y\in}V$ and $\mathbf{z}\in H$).\medskip

Let $U$ be a convex, closed subset of an Euclidian space $\mathbb{R}^{m}$.
On the coefficients of the equation we impose the following conditions:

\begin{description}
\item[(A$_{0}$)] $\mathbf{a}:\mathcal{O}\times\mathbb{R}^{n}\rightarrow
\mathbb{R}^{n}$ is a Carath\'{e}odory function\footnote{%
i.e., $\mathbf{a}(x,\cdot)$ is continuous for every $x\in\mathcal{O}$ and $%
\mathbf{a}(\cdot,\zeta)$ is measurable for every $\zeta\in\mathbb{R}^{n}$}
with $\mathbf{a}(x,\cdot)\in C^{1}(\mathbb{R}^{n};\mathbb{R}^{n})$, $dx$%
-a.e. on $\mathcal{O}$;

\item there exist constants $\delta,c_{0}>0$ and positive functions $\rho\in
L^{2}(\mathcal{O})$, $\tilde{\rho}\in L^{2}(\Gamma)$ such that:

\item[(A$_{1}$)] for almost all $x\in\mathcal{O}$ and all $\zeta\in \mathbb{R%
}^{n}$:%
\begin{align*}
\left\vert \mathbf{a}(x,\zeta)\right\vert & \leq c_{0}(\rho(x)+\left\vert
\zeta\right\vert ), \\
\left\vert D_{\zeta}\mathbf{a}(x,\zeta)\right\vert & \leq c_{0};
\end{align*}

\item[(A$_{2}$)] for almost all $x\in\mathcal{O}$ and all $\zeta\in \mathbb{R%
}^{n}$:%
\begin{equation*}
\left( \mathbf{a}(x,\zeta)-\mathbf{a}(x,\eta)\right) \cdot(\zeta-\eta
)\geq\delta\left\vert \zeta-\eta\right\vert ^{2};
\end{equation*}

\item[(B)] $b:L^{2}(\mathcal{O})\rightarrow L^{2}(\mathcal{O})$ and $\tilde {%
b}:L^{2}(\Gamma)\rightarrow L^{2}(\Gamma)$ are Hilbert-Schmidt linear
operators;

\item[(C$_{0}$)] $\gamma:\Gamma\times\mathbb{R}\times U\rightarrow\mathbb{R}$
is a Carath\'{e}odory function with $\gamma(\xi,\cdot)\in C^{1}(\mathbb{R}%
\times U)$, $d\xi$-a.e. on $\Gamma$;

\item[(C$_{1}$)] for almost all $\xi\in\Gamma$ and all $\left( \bar {y}%
,u\right) \in\mathbb{R}\times U$:%
\begin{align*}
\left\vert \gamma(\xi,\bar{y},u)\right\vert & \leq c_{0}\left( \tilde{\rho }%
(\xi)+\left\vert \bar{y}\right\vert +\left\vert u\right\vert \right) , \\
\left\vert D_{\bar{y}}\gamma(\xi,\bar{y},u)\right\vert & \leq c_{0}, \\
\left\vert D_{u}\gamma(\xi,\bar{y},u)\right\vert & \leq c_{0}\left( \tilde{%
\rho}(\xi)+\left\vert \bar{y}\right\vert \right) ;
\end{align*}

\item[(C$_{2}$)] for almost all $\xi\in\Gamma$ and all $\left( \bar {y}%
,u\right) \in\mathbb{R}\times U$:%
\begin{equation*}
\left( \gamma(\xi,\bar{y},u)-\gamma(\xi,\bar{y}^{\prime},u)\right) \left(
\bar{y}-\bar{y}^{\prime}\right) \geq\delta\left\vert \bar{y}-\bar{y}^{\prime
}\right\vert ^{2}.
\end{equation*}
\end{description}

In order to give a functional setting for our equation, let $S$ be the space of smooth functions
\begin{equation*}
S:=\left\{ (y,\bar{y})\in C^{\infty
}(\mathcal{O})\times C^{\infty}(\Gamma)\mid\bar{y}=y|_\Gamma\right\}.
\end{equation*}
We define an
operator $A:S\times L^{2}(\Gamma;U)\rightarrow H\subseteq V^{\ast}$ by%
\begin{equation*}
A(y,\bar{y},u):=%
\begin{pmatrix}
\func{div}\mathbf{a}(x,\nabla y) \\
-\mathbf{a}(\xi,\nabla y)\cdot\nu-\gamma(\xi,y(t,\xi),u(t,\xi))%
\end{pmatrix}
,\ (y,\bar{y},u)\in S\times L^{2}(\Gamma;U).
\end{equation*}
An integration by parts, hypotheses (A$_{0}$), (A$_{1}$%
), and the density of $S$ in $V$ show that $A$ can be extended to a bounded
non-linear operator on $V\times L^{2}(\Gamma;U)$ with values in $V^{\ast}$,
again denoted by $A$, such that%
\begin{equation*}
_{V^{\ast}}\left\langle A(y,\bar{y},u),(z,\bar{z})\right\rangle _{V}=-\int_{%
\mathcal{O}}\mathbf{a}(x,\nabla y)\cdot\nabla z\,dx-\int_{\Gamma }\gamma(\xi,%
\bar{y},u)\bar{z}\,d\xi,
\end{equation*}
for all $(y,\bar{y}),(z,\bar{z})\in V$ and $u\in L^{2}(\Gamma;U)$.

We also set $B:=%
\begin{pmatrix}
b & 0 \\
0 & \tilde{b}%
\end{pmatrix}
$, so that $B$ is a Hilbert-Schmidt operator; we denote $L_{2}(H)$ the space
of such operators, endowed with the norm $\left\Vert T\right\Vert
_{L_{2}(H)}:=\left( \sum_{j=1}^{\infty}\left\vert Te_{j}\right\vert
^{2}\right) ^{1/2}$, for an orthonormal basis $\left( e_{j}\right) $ of $H$.
Consider a $H$-cylindrical Wiener process, formally written%
\begin{equation*}
W(t)=%
\begin{pmatrix}
w(t) \\
\tilde{w}(t)%
\end{pmatrix}
:=%
\begin{pmatrix}
\sum_{k=1}^{\infty}\beta_{k}^{1}(t)g_{k}^{1} \\
\sum_{k=1}^{\infty}\beta_{k}^{2}(t)g_{k}^{2}%
\end{pmatrix}
,
\end{equation*}
where $\left\{ g_{k}^{1}\right\} $ and $\left\{ g_{k}^{2}\right\} $ are
orthonormal bases in $L^{2}(\mathcal{O})$ and $L^{2}(\Gamma)$, respectively,
$\left\{ \beta_{k}^{i}\right\} _{k=1,\infty}^{i=1,2}$ is a sequence of
independent Brownian motions on $\left( \Omega,\mathcal{F},\mathbb{P}\right)
$ and $\left\{ \mathcal{F}_{t}\right\} _{t\geq0}$ is the filtration
generated by $\left\{ \beta_{k}^{i}\right\} _{k=1,\infty}^{i=1,2}$,
augmented by the null sets of $\mathcal{F}$.

Then, for $\mathbf{y}_{0}=\left( y_{0},\bar{y}_{0}\right) \in H$, the
state equation (\ref{Eq_state0}) can be written as%
\begin{equation}
\mathbf{Y}(t)=\mathbf{y}_{0}+\int_{0}^{t}A(\mathbf{Y}(s),u(s)) \,
ds+\int_{0}^{t}B \, dW(s),\ t\in\lbrack0,T].   \label{Eq_state}
\end{equation}
\bigskip

Here we assume that $u$ is an \emph{admissible control} (or simply,
control), \textit{i.e.} a progressively measurable process $u\in L^{2}\left(
\Omega\times\lbrack0,T];L^{2}(\Gamma;U)\right) $. We will denote by $%
\mathcal{U}$ the space of all admissible controls.

\begin{theorem}
\label{Th_Ex_state_eq}Under hypotheses (A$_{0}$)--(A$_{2}$), (B), (C$_{0}$%
)--(C$_{2}$), for every control $u$, there exists a unique solution $\mathbf{%
Y}^{u}=(Y^{u},\bar{Y}^{u})\in L^{2}\left( \Omega\times \lbrack0,T];V\right) $
of equation (\ref{Eq_state}) such that $\mathbf{Y}^{u}$ is a continuous,
adapted process with values in $H$. Moreover,%
\begin{equation*}
\mathbb{E}\sup_{t\in\left[ 0,T\right] }\left\Vert \mathbf{Y}%
^{u}(t)\right\Vert _{H}^{2}<+\infty.
\end{equation*}
\end{theorem}

For the proof of this result the reader can refer to the book of Pr\'{e}v%
\^{o}t and R\"{o}ckner \cite{PreRoc05}, where a general result of existence
and uniqueness for variational solutions was given. The task of verifying
that the above hypotheses are sufficient to place ourselves into their
framework was already carried in \cite{BonZig14}.

The notion of solution for \eqref{Eq_state} that is used in Theorem \ref%
{Th_Ex_state_eq} is that of \emph{variational solution} as given in the book
by Pr\'{e}v\^{o}t and R\"{o}ckner \cite[Definition 4.2.1]{PreRoc05}.
Actually, this means that $X$ is an $H$-valued, adapted process with an
equivalent version that belongs to $L^2\left( \Omega\times
\lbrack0,T];V\right)$ and satisfies the equation $\mathbb{P}$-a.s.

Concerning the cost functional (\ref{Cost_functional0}), on its coefficients
we impose the following hypotheses (the functions $\rho$ and $\tilde{\rho}$
were already introduced for the previous set of conditions):

\begin{description}
\item[(F$_{0}$)] $\psi:\mathcal{O}\times\mathbb{R}$ and $\bar{\psi}%
:\Gamma\times\mathbb{R}$ are Carath\'{e}odory functions with $%
\psi(x,\cdot)\in C^{1}(\mathbb{R})$, $dx$-a.e. on $\mathcal{O}$ and $\bar{%
\psi}(\xi,\cdot)\in C^{1}(\mathbb{R})$, $d\xi$-a.e. on $\Gamma$;

\item there exist constants $c_{1},c_{2}>0$ such that:

\item[(F$_{1}$)] for almost all $x\in\mathcal{O}$ and all $y\in\mathbb{R}$:%
\begin{align*}
\left\vert \psi(x,y)\right\vert & \leq c_{1}(\rho(x)^{2}+\left\vert
y\right\vert ^{2}), \\
\left\vert D_{y}\psi(x,y)\right\vert & \leq c_{1}(\rho(x)+\left\vert
y\right\vert );
\end{align*}
for almost all $\xi\in\Gamma$ and all $\bar{y}\in\mathbb{R}$:%
\begin{align*}
\left\vert \bar{\psi}(\xi,\bar{y})\right\vert & \leq c_{1}(\tilde{\rho}%
(\xi)^{2}+\left\vert \bar{y}\right\vert ^{2}), \\
\left\vert D_{\bar{y}}\bar{\psi}(\xi,\bar{y})\right\vert & \leq c_{1}(\tilde{%
\rho}(\xi)+\left\vert \bar{y}\right\vert );
\end{align*}

\item[(L$_{0}$)] $\ell:\mathcal{O}\times\mathbb{R}$ and $\bar{\ell}%
:\Gamma\times\mathbb{R}\times U$ are Carath\'{e}odory functions with $%
\ell(x,\cdot)\in C^{1}(\mathbb{R})$, $dx$-a.e. on $\mathcal{O}$ and $\bar{%
\ell}(\xi,\cdot,\cdot)\in C^{1}(\mathbb{R}\times U)$, $d\xi$-a.e. on $\Gamma$%
;

\item[(L$_{1}$)] for almost all $x\in\mathcal{O}$ and all $y\in\mathbb{R}$:%
\begin{align*}
\left\vert \ell(x,y)\right\vert & \leq c_{2}(\rho(x)^{2}+\left\vert
y\right\vert ^{2}), \\
\left\vert D_{y}\ell(x,y)\right\vert & \leq c_{2}(\rho(x)+\left\vert
y\right\vert ),
\end{align*}
for almost all $\xi\in\Gamma$ and all $(\bar{y},u)\in\mathbb{R}\times U$:%
\begin{align*}
\left\vert \bar{\ell}(\xi,\bar{y},u)\right\vert & \leq c_{2}(\tilde{\rho }%
(\xi)^{2}+\left\vert \bar{y}\right\vert ^{2}+\left\vert u\right\vert ^{2}),
\\
\left\vert D_{\bar{y}}\bar{\ell}(\xi,\bar{y},u)\right\vert & \leq c_{2}(%
\tilde{\rho}(\xi)+\left\vert \bar{y}\right\vert +\left\vert u\right\vert ),
\\
\left\vert D_{u}\bar{\ell}(\xi,\bar{y},u)\right\vert & \leq c_{2}(\tilde{\rho%
}(\xi)+\left\vert \bar{y}\right\vert +\left\vert u\right\vert ).
\end{align*}
\end{description}

The cost functional can then be written as%
\begin{equation}
J(u):=\mathbb{E}\left[ \int_{0}^{T}L(\mathbf{Y}^{u}(t),u(t))dt+\Psi (\mathbf{%
Y}^{u}(T))\right] ,   \label{Cost_functional}
\end{equation}
where $L:H\times L^{2}(\Gamma;U)\rightarrow\mathbb{R}$ and $\Psi
:H\rightarrow\mathbb{R}$ are defined by%
\begin{align*}
L(y,\bar{y},u) & :=\int_{\mathcal{O}}\ell(x,y(x))dx+\int_{\Gamma}\bar{\ell }%
(\xi,\bar{y}(\xi),u(\xi))d\xi; \\
\Psi(y,\bar{y}) & :=\int_{\mathcal{O}}\psi(x,y(x))dx+\int_{\Gamma}\bar{\psi }%
(\xi,\bar{y}(\xi))d\xi.
\end{align*}

From now on, we will assume that conditions (A$_{0}$)--(A$_{2}$), (B), (C$%
_{0}$)--(C$_{2}$), (F$_{0}$), (F$_{1}$), (L$_{0}$) and (L$_{1}$) are in
force.

It is easy to show that $\Psi$ and $L$ are G\^{a}teaux differentiable in $%
\mathbf{y}=(y,\bar{y})\in H$, with%
\begin{align*}
D_{\mathbf{y}}\Psi(\mathbf{y}) & =\left( D_{y}\psi(\cdot,y(\cdot )),D_{\bar{y%
}}\bar{\psi}(\cdot,\bar{y}(\cdot))\right) ; \\
D_{\mathbf{y}}L(\mathbf{y},u) & =\left( D_{y}\ell(\cdot,y(\cdot)),D_{\bar {y}%
}\bar{\ell}(\cdot,\bar{y}(\cdot),u(\cdot))\right) .
\end{align*}
Also, $A$ is G\^{a}teaux differentiable in $\mathbf{y}=(y,\bar{y})\in V$,
with%
\begin{equation*}
_{V^{\ast}}\left\langle \left( D_{\mathbf{y}}A\right) (\mathbf{y},u)(p,\bar{p%
}),(z,\bar{z})\right\rangle _{V}=-\int_{\mathcal{O}}D_{\zeta }\mathbf{a}%
(x,\nabla y)\nabla p\cdot\nabla z\,dx-\int_{\Gamma}D_{\bar{y}}\gamma(\xi,%
\bar{y},u)\bar{p}\bar{z}\,d\xi,\ (p,\bar{p}),(z,\bar{z})\in V.
\end{equation*}

\section{Maximum principle}

\subsection{The adjoint equation}

We consider the following linear BSDE in $V^{\ast}$:%
\begin{equation}
\mathbf{P}^{u}(t)=D_{\mathbf{y}}\Psi(\mathbf{Y}^{u}(T))+\int_{t}^{T}\left(
D_{\mathbf{y}}A\right) ^{\ast}(\mathbf{Y}^{u}(s),u(s))\mathbf{P}^{u}(s) \,
ds+\int_{t}^{T}D_{\mathbf{y}}L(\mathbf{Y}^{u}(s),u(s)) \,
ds-\int_{t}^{T}Q^{u}(s) \, dW(s).   \label{Eq_Adjoint}
\end{equation}

\begin{theorem}
For every control $u$, there exists a unique solution $(\mathbf{P}%
^{u},Q^{u})=(P^{u},\bar{P}^{u},Q^{u})\in L^{2}\left( \Omega\times\lbrack
0,T];V\right) \times L^{2}\left( \Omega\times\lbrack0,T];L_{2}(H)\right) $
such that $\mathbf{P}^{u}$ is a continuous, adapted process with values in $H
$.
\end{theorem}

\begin{proof}
In order to prove this theorem, we will use a result of M\'{a}rquez-Dur\'{a}%
n and Real \cite{MarDur04} which asserts existence and uniqueness for
general (non-linear) BSDEs in a variational setting. Let us now verify that
the hypotheses of Theorem 2.2 in \cite{MarDur04} are fulfilled for the
coefficients of our BSDE.

\begin{enumerate}
\item \emph{Final condition.} The fact that $D_{\mathbf{y}}\Psi(\mathbf{Y}%
^{u}(T))\in L^{2}(\Omega,\mathcal{F}_{T},P;H)$ is clearly implied by linear
growth condition on $D_{y}\psi$ and $D_{\bar{y}}\bar{\psi}$.

\item \emph{Measurability.} Of course,%
\begin{equation*}
\left( D_{\mathbf{y}}A\right) ^{\ast}(\mathbf{Y}^{u},u)\mathbf{p}+D_{\mathbf{%
y}}L(\mathbf{Y}^{u},u)
\end{equation*}
is a progressively measurable process with values in $V^{\ast}$ for every $(%
\mathbf{p},q)\in V\times L_{2}(H)$.

\item \emph{Hemicontinuity.} The mapping%
\begin{equation*}
\lambda\longmapsto_{V^{\ast}}\left\langle \left( D_{\mathbf{y}}A\right)
^{\ast}(\mathbf{Y}^{u}(t),u(t))(\mathbf{p}+\lambda\mathbf{p}^{\prime }),%
\mathbf{z}\right\rangle _{V}
\end{equation*}
is continuous, for every $(t,\mathbf{p},\mathbf{p}^{\prime})\in\lbrack
0,T]\times V\times V$ and $\mathbf{z}\in V$. Indeed, for $\mathbf{p}=(p,\bar{%
p}),\ \mathbf{p}^{\prime}=(p^{\prime},\bar{p}^{\prime})$ and $\mathbf{z}=(z,%
\bar{z})$, we have
\begin{align*}
_{V^{\ast}}\left\langle \left( D_{\mathbf{y}}A\right) ^{\ast}(\mathbf{Y}%
^{u}(t),u(t))(\mathbf{p}+\lambda\mathbf{p}^{\prime}),\mathbf{z}\right\rangle
_{V} & =-\int_{\mathcal{O}}D_{\zeta}\mathbf{a}(x,\nabla Y^{u}(t))\nabla
z\cdot\left( \nabla p+\lambda\nabla p^{\prime}\right) dx \\
& -\int_{\Gamma}D_{\bar{y}}\gamma(\xi,\bar{Y}^{u}(t),u(t))\bar{z}\left( \bar{%
p}+\lambda\bar{p}^{\prime}\right) d\xi
\end{align*}
and the conclusion follows from the Lebesgue's dominated convergence
theorem, by (A$_{1}$) and (C$_{1}$).

\item \emph{Boundedness.} By (L$_{1}$), $D_{\mathbf{y}}L(\mathbf{Y}%
^{u}(\cdot),u(\cdot))\in L^{2}\left( \Omega\times\lbrack0,T];H\right) $.
Moreover, by (A$_{1}$) and (C$_{1}$), for every $(t,\mathbf{p},q)\in
\lbrack0,T]\times V\times L_{2}(H)$, $\left\Vert \left( D_{\mathbf{y}%
}A\right) ^{\ast}(\mathbf{Y}^{u}(t),u(t))\mathbf{p}\right\Vert _{V^{\ast}}$
is bounded by $c_{0}$.

\item \emph{Monotonicity.} We have that%
\begin{equation*}
_{V^{\ast}}\left\langle \left( D_{\mathbf{y}}A\right) ^{\ast}(\mathbf{Y}%
^{u}(t),u(t))\mathbf{p},\mathbf{p}\right\rangle _{V}=-\int_{\mathcal{O}%
}D_{\zeta}\mathbf{a}(x,\nabla Y^{u})\nabla p\cdot\nabla p\,dx-\int_{\Gamma
}D_{\bar{y}}\gamma(\xi,\bar{Y}^{u},u(t))\left\vert \bar{p}\right\vert
^{2}\,d\xi\leq0,
\end{equation*}
for every $(\mathbf{p},q)=(p,\bar{p},q)\in V\times L_{2}(H)$, $d\mathbb{P}%
\times dt$ a.e., by assumptions (A$_{2}$) and (C$_{2}$).

\item \emph{Coercivity.} There exist $\alpha>0$, $\lambda\in\mathbb{R}$ and
a progressively measurable process $C(\cdot)\in L^{1}\left( \Omega\times
\lbrack0,T]\right) $ such that%
\begin{equation*}
-\ _{V^{\ast}}\left\langle \left( D_{\mathbf{y}}A\right) ^{\ast}(\mathbf{Y}%
^{u}(t),u(t))\mathbf{p},\mathbf{p}\right\rangle _{V}-\left\langle D_{\mathbf{%
y}}L(\mathbf{Y}^{u}(t),u(t)),\mathbf{p}\right\rangle _{H}+\lambda\left\Vert
\mathbf{p}\right\Vert _{H}^{2}+C(t)\geq\alpha\left\Vert \mathbf{p}%
\right\Vert _{V}^{2},
\end{equation*}
for every $(\mathbf{p},q)\in V\times L_{2}(H)$, $d\mathbb{P}\times dt$ a.e.
Indeed, for $\mathbf{p}=(p,\bar{p})$, we have%
\begin{align*}
-\ _{V^{\ast}}\left\langle \left( D_{\mathbf{y}}A\right) ^{\ast}(\mathbf{Y}%
^{u}(t),u(t))\mathbf{p},\mathbf{p}\right\rangle _{V} & -\left\langle D_{%
\mathbf{y}}L(\mathbf{Y}^{u}(t),u(t)),\mathbf{p}\right\rangle _{H}=\int_{%
\mathcal{O}}D_{\zeta}\mathbf{a}(x,\nabla Y^{u}(t))\nabla p\cdot\nabla p\,dx
\\
& +\int_{\Gamma}D_{\bar{y}}\gamma(\xi,\bar{Y}^{u}(t),u(t))\left\vert \bar {p}%
\right\vert ^{2}d\xi-\left\langle D_{\mathbf{y}}L(\mathbf{Y}^{u}(t),u(t)),%
\mathbf{p}\right\rangle _{H} \\
\geq & \delta\left( \left\Vert \nabla p\right\Vert _{L^{2}\left( \mathcal{O}%
\right) }^{2}+\left\Vert \bar{p}\right\Vert _{L^{2}\left( \Gamma\right)
}^{2}\right) -\frac{1}{2}\left\Vert D_{\mathbf{y}}L(\mathbf{Y}%
^{u}(t),u(t))\right\Vert _{H}^{2}-\frac{1}{2}\left\Vert \mathbf{p}%
\right\Vert _{H}^{2} \\
\geq & \delta\left\Vert \mathbf{p}\right\Vert _{V}^{2}-\frac{1}{2}\left\Vert
\mathbf{p}\right\Vert _{H}^{2}-\frac{3}{2}c_{2}^{2}\left( \left\Vert (\rho,%
\tilde{\rho})\right\Vert _{H}^{2}+\left\Vert \mathbf{Y}^{u}(t)\right\Vert
_{H}^{2}+\left\Vert u(t)\right\Vert _{L^{2}\left( \Gamma\right) }^{2}\right)
.
\end{align*}
\end{enumerate}
\end{proof}

\subsection{The variational equation}

We define the operator $\mathcal{G}:L^{2}\left( \Gamma\right) \times
L^{2}(\Gamma;U)\times L^{\infty}(\Gamma;U)\rightarrow H$ by%
\begin{equation*}
\mathcal{G}(\bar{y},u,\bar{u}):=\left( 0,-D_{u}\gamma(\cdot,\bar{y},u)\cdot%
\bar{u}\right) .
\end{equation*}

Let now $u$ and $v$ be two controls such that $v-u$ is bounded; let, for $%
\theta\in\lbrack0,1]$, $u^{\theta}:=(1-\theta)u+\theta v$. Let us denote,
for simplicity, $\mathbf{Y}^{\theta}$, $Y^{\theta}$ and $\bar{Y}^{\theta}$
instead of $\mathbf{Y}^{u^{\theta}}$, $Y^{u^{\theta}}$ and $\bar{Y}%
^{u^{\theta}}$, respectively.

\begin{proposition}
\label{P_var_eq}The equation%
\begin{equation}
\mathbf{Z}(t)=\int_{0}^{t}D_{\mathbf{y}}A(\mathbf{Y}^{u}(s),u(s))\mathbf{Z}%
(s)ds+\int_{0}^{t}\mathcal{G}(\bar{Y}^{u}(s),u(s),v(s)-u(s))ds,\ t\in\left[
0,T\right]   \label{Eq_Var}
\end{equation}
has a unique variational solution $\mathbf{Z}$ that is a continuous, adapted
process in $H$ with $\mathbf{Z}\in L^{2}\left( \Omega
\times\lbrack0,T];V\right) $. Moreover, $\frac{1}{\theta}\left( \mathbf{Y}%
^{\theta}-\mathbf{Y}^{0}\right) $ and $\frac{1}{\theta}\left( \mathbf{Y}%
^{\theta}(T)-\mathbf{Y}^{0}(T)\right) $ converge weakly\footnote{%
Recall that a sequence $\left( Z^{\theta}\right) $ of random variables
taking values in a Hilbert space $X$ converges weakly to $Z$ in $%
L^{2}(\Omega,X)$ as $\theta\rightarrow0$ if for any random variable $\bar{Z}\in L^{2}(\Omega,X)$
we have $\mathbb{E}\left\langle Z^{\theta},\bar{Z}\right\rangle \rightarrow
\mathbb{E}\left\langle Z,\bar{Z}\right\rangle $.} as $\theta\rightarrow0$ to $%
\mathbf{Z}$ and $\mathbf{Z}(T)$ in $L^{2}\left( \Omega\times\lbrack
0,T];V\right) $, respectively in $L^{2}\left( \Omega;H\right) $.
\end{proposition}

\begin{proof}
We have, by It\^{o}'s formula,%
\begin{align*}
\mathbb{E}\left\Vert \mathbf{Y}^{\theta}(t)\right\Vert _{H}^{2} &
=\left\Vert \mathbf{y}_{0}\right\Vert _{H}^{2}-2\mathbb{E}\left[
\int_{0}^{t}\int_{\mathcal{O}}\mathbf{a}(x,\nabla Y^{\theta}(s))\cdot\nabla
Y^{\theta }(s)\,dx\,ds\right] \\
& -2\mathbb{E}\left[ \int_{0}^{t}\int_{\Gamma}\gamma(\xi,\bar{Y}^{\theta
}(s),u^{\theta}(s))\bar{Y}^{\theta}(s)\,d\xi\,ds\right] +t\left\Vert
B\right\Vert _{L_{2}(H)}^{2},\ t\in\lbrack0,T],
\end{align*}
therefore, by (A$_{1}$), (A$_{2}$), (C$_{1}$) and (C$_{2}$),%
\begin{equation}
\sup_{\theta\in\lbrack0,1]}\left[ \sup_{t\in\lbrack0,T]}\mathbb{E}\left\Vert
\mathbf{Y}^{\theta}(t)\right\Vert _{H}^{2}+\mathbb{E}\int_{0}^{T}\left\Vert
\mathbf{Y}^{\theta}(t)\right\Vert _{V}^{2}dt\right] <+\infty .
\label{rel_bound}
\end{equation}
Since, for any $t\in\lbrack0,T]$,%
\begin{align*}
\left\Vert \left( \mathbf{Y}^{\theta}(t)-\mathbf{Y}^{0}(t)\right)
\right\Vert _{H}^{2}= & -2\int_{0}^{t}\int_{\mathcal{O}}\left[ \mathbf{a}%
(x,\nabla Y^{\theta}(s))-\mathbf{a}(x,\nabla Y^{0}(s))\right] \cdot\left[
\nabla Y^{\theta}(s)-\nabla Y^{0}(s)\right] \,dx\,ds \\
& -2\int_{0}^{t}\int_{\Gamma}\left[ \gamma(\xi,\bar{Y}^{\theta}(s),u^{%
\theta}(s))-\gamma(\xi,\bar{Y}^{0}(s),u(s))\right] \left[ \bar {Y}%
^{\theta}(s)-\bar{Y}^{0}(s)\right] \,d\xi\,ds,
\end{align*}
we have, by the assumptions on $\mathbf{a}$ and $\gamma$,%
\begin{align*}
& \left\Vert \mathbf{Y}^{\theta}(t)-\mathbf{Y}^{0}(t)\right\Vert
_{H}^{2}+2\delta\int_{0}^{t}\left\Vert \mathbf{Y}^{\theta}(s)-\mathbf{Y}%
^{0}(s)\right\Vert _{V}^{2}\,ds \\
& \leq-2\int_{0}^{t}\int_{\Gamma}\left[ \gamma(\xi,\bar{Y}^{0}(s),u^{\theta
}(s))-\gamma(\xi,\bar{Y}^{0}(s),u(s))\right] \left( \bar{Y}^{\theta}(s)-\bar{%
Y}^{0}(s)\right) \,d\xi\,ds \\
& =-2\theta\int_{0}^{t}\int_{\Gamma}\left[ \int_{0}^{1}D_{u}\gamma(\xi ,\bar{%
Y}^{0}(s),u^{\lambda\theta}(s))d\lambda\right] (v(s)-u(s))\left( \bar{Y}%
^{\theta}(s)-\bar{Y}^{0}(s)\right) \,d\xi\,ds \\
& \leq C\theta\int_{0}^{t}\int_{\Gamma}\left[ \tilde{\rho}(\xi)+\left\vert
\bar{Y}^{0}(s)\right\vert \right] \left( \bar{Y}^{\theta}(s)-\bar{Y}%
^{0}(s)\right) \,d\xi\,ds \\
& \leq C\theta^{2}\int_{0}^{t}\left( \tilde{\rho}(\xi)^{2}+\left\Vert
\mathbf{Y}^{0}(s)\right\Vert _{V}^{2}\right) ds+\delta\int_{0}^{t}\left\Vert
\mathbf{Y}^{\theta}(s)-\mathbf{Y}^{0}(s)\right\Vert _{V}^{2}\,ds,
\end{align*}
where \ $C>0$ is a constant whose value is allowed to change from line to
line. Hence
\begin{equation}
\mathbb{E}\left[ \sup_{\theta\in\lbrack0,1]}\sup_{t\in\lbrack0,T]}\left\Vert
\tfrac{1}{\theta}(\mathbf{Y}^{\theta}(t)-\mathbf{Y}^{0}(t))\right\Vert
_{H}^{2}\right] +\mathbb{E}\left[ \sup_{\theta\in\lbrack0,1]}\int_{0}^{T}%
\left\Vert \tfrac{1}{\theta}(\mathbf{Y}^{\theta}(t)-\mathbf{Y}%
^{0}(t))\right\Vert _{V}^{2}\,ds\right] <+\infty.   \label{rel_bound_deriv}
\end{equation}
Then there exists a progressively measurable process $\mathbf{\tilde{Z}}\in
L^{2}\left( \Omega\times\lbrack0,T];V\right) $ such that, at least on a
subsequence:

\begin{itemize}
\item $\tfrac{1}{\theta}(\mathbf{Y}^{\theta}-\mathbf{Y}^{0})$ converges
weakly to $\mathbf{\tilde{Z}}$ as $\theta\rightarrow0$ in $L^{2}\left(
\Omega \times\lbrack0,T];V\right) $;

\item $\nabla Y^{\theta}$ converges to $\nabla Y^{0}$ a.e. as $\theta
\rightarrow0$ on $\Omega\times\lbrack0,T]\times\mathcal{O}$;

\item $\bar{Y}^{\theta}$ converges to $\bar{Y}^{0}$ a.e. as $\theta
\rightarrow0$ on $\Omega\times\lbrack0,T]\times\Gamma$.
\end{itemize}

For $\mathbf{z}=(z,\bar{z})\in V$ with $z\in C_{b}^{1}\left( \mathcal{O}%
\right) $, we have%
\begin{align}
\left\langle \frac{\mathbf{Y}^{\theta}(t)-\mathbf{Y}^{0}(t)}{\theta },%
\mathbf{z}\right\rangle _{H}= & -\int_{0}^{t}\int_{\mathcal{O}}\frac{\mathbf{%
a}(x,\nabla Y^{\theta}(s))-\mathbf{a}(x,\nabla Y^{0}(s))}{\theta}\cdot\nabla
z\,dx\,ds  \label{rel_weak_theta} \\
& -\int_{0}^{t}\int_{\Gamma}\frac{\gamma(\xi,\bar{Y}^{\theta}(s),u^{\theta
}(s))-\gamma(\xi,\bar{Y}^{0}(s),u(s))}{\theta}\bar{z}\,d\xi\,ds  \notag \\
= & -\int_{0}^{t}\int_{\mathcal{O}}T^{1,\theta}(s)\frac{\nabla(Y^{\theta
}-Y^{0})(s)}{\theta}\cdot\nabla zdxds-\int_{0}^{t}\int_{\Gamma}T^{2,\theta
}(s)\frac{(\bar{Y}^{\theta}-\bar{Y}^{0})(s)}{\theta}\bar{z}\,d\xi \,ds
\notag \\
& -\int_{0}^{t}\int_{\Gamma}T^{3,\theta}(s)(v(s)-u(s))\bar{z}d\xi
ds-\int_{\Gamma}D_{u}\gamma(\xi,\bar{Y}^{0}(s),u(s))(v(s)-u(s))\bar{z}\,d\xi
\notag \\
& -\int_{0}^{t}\int_{\mathcal{O}}D_{\zeta}\mathbf{a}(x,\nabla Y^{0}(s)))%
\frac{\nabla(Y^{\theta}-Y^{0})(s)}{\theta}\cdot\nabla z\,dx\,ds  \notag \\
& -\int_{0}^{t}\int_{\Gamma}D_{\bar{y}}\gamma(\xi,\bar{Y}^{0}(s),u(s))\frac {%
(\bar{Y}^{\theta}-\bar{Y}^{0})(s)}{\theta}\bar{z}\,d\xi\,ds,  \notag
\end{align}
where, for the sake of simplicity, we have denoted
\begin{align*}
T^{1,\theta}(s) & :=\int_{0}^{1}\left[ D_{\zeta}\mathbf{a}(x,\nabla
Y^{0}(s)+\lambda\nabla(Y^{\theta}-Y^{0})(s))-D_{\zeta}\mathbf{a}(x,\nabla
Y^{0}(s))\right] d\lambda; \\
T^{2,\theta}(s) & :=\int_{0}^{1}\left[ D_{\bar{y}}\gamma(\xi,\bar{Y}%
^{0}(s)+\lambda(\bar{Y}^{\theta}-\bar{Y}^{0})(s),u^{\lambda\theta}(s))-D_{%
\bar{y}}\gamma(\xi,\bar{Y}^{0}(s),u(s))\right] d\lambda; \\
T^{3,\theta}(s) & :=\int_{0}^{1}\left[ D_{u}\gamma(\xi,\bar{Y}%
^{0}(s)+\lambda(\bar{Y}^{\theta}-\bar{Y}^{0})(s),u^{\lambda\theta}(s))-D_{u}%
\gamma(\xi,\bar{Y}^{0}(s),u(s))\right] d\lambda.
\end{align*}
By the dominated convergence theorem and (\ref{rel_bound_deriv}), since $%
T^{1,\theta}$ and $T^{2,\theta}$ are bounded, we have that%
\begin{equation*}
\lim_{\theta\rightarrow0}\mathbb{E}\left[ \int_{0}^{T}\int_{\mathcal{O}%
}\left\vert T^{1,\theta}(s)\right\vert \left\vert \frac{\nabla(Y^{\theta
}-Y^{0})(s)}{\theta}\right\vert \left\vert \nabla z\right\vert \,dx\,ds%
\right] =0
\end{equation*}
and%
\begin{equation*}
\lim_{\theta\rightarrow0}\mathbb{E}\left[ \int_{0}^{t}\int_{\Gamma}\left%
\vert T^{2,\theta}(s)\right\vert \left\vert \frac{(\bar{Y}^{\theta}-\bar{Y}%
^{0})(s)}{\theta}\right\vert \left\vert \bar{z}\right\vert \,d\xi\,ds\right]
=0.
\end{equation*}
We also have that%
\begin{align*}
\mathbb{E}\left[ \int_{0}^{t}\int_{\Gamma}\left\vert T^{3,\theta
}(s)\right\vert \left\vert (v(s)-u(s))\right\vert \left\vert \bar {z}%
\right\vert \,d\xi\,ds\right] \leq & C\mathbb{E}\left[ \int_{0}^{T}\int_{%
\Gamma}\left\vert T^{3,\theta}(s)\right\vert \rho_{1}\left( (\bar {Y}%
^{\theta}-\bar{Y}^{0})(s)\right) \,d\xi\,ds\right] \\
& +C\mathbb{E}\left[ \int_{0}^{T}\int_{\Gamma}\left\vert T^{3,\theta
}(s)\right\vert \left( 1-\rho_{1}\right) \left( (\bar{Y}^{\theta}-\bar {Y}%
^{0})(s)\right) \,d\xi\,ds\right] ,
\end{align*}
where $\rho_{1}$ is a smooth function defined on $\mathbb{R}$ such that $%
0\leq\rho_{1}\leq1$, $\rho_{1}(\bar{y})=1$ for $\left\vert \bar{y}%
\right\vert \leq1$ and $\rho_{1}(\bar{y})=0$ for $\left\vert \bar{y}%
\right\vert \geq2$. Since, by (C$_{1}$),%
\begin{equation*}
\left\vert T^{3,\theta}(s)\right\vert \rho_{1}\left( (\bar{Y}^{\theta}-\bar{Y%
}^{0})(s)\right) \leq C\left( \tilde{\rho}(\xi)+\left\vert \bar {Y}%
^{0}(s)\right\vert \right) ,
\end{equation*}
we have, by the dominated convergence theorem, that%
\begin{equation*}
\lim_{\theta\rightarrow0}\mathbb{E}\left[ \int_{0}^{T}\int_{\Gamma}\left%
\vert T^{3,\theta}(s)\right\vert \rho_{1}\left( (\bar{Y}^{\theta}-\bar{Y}%
^{0})(s)\right) \,d\xi\,ds\right] =0.
\end{equation*}
On the other hand, by (\ref{rel_bound}) and (\ref{rel_bound_deriv}),%
\begin{align*}
& \mathbb{E}\left[ \int_{0}^{T}\int_{\Gamma}\left\vert T^{3,\theta
}(s)\right\vert \left( 1-\rho_{1}\right) \left( (\bar{Y}^{\theta}-\bar {Y}%
^{0})(s)\right) \,d\xi\,ds\right] \\
& \leq C\mathbb{E}\left[ \int_{0}^{T}\int_{\Gamma}\left( \tilde{\rho}%
(\xi)+\left\vert \bar{Y}^{0}(s)\right\vert +\left\vert \bar{Y}^{\theta
}(s)\right\vert \right) \mathbf{1}_{\left\{ \left\vert (\bar{Y}^{\theta }-%
\bar{Y}^{0})(s)\right\vert \geq1\right\} }\,d\xi\,ds\right] \\
& \leq C\left( \mathbb{E}\left[ \int_{0}^{T}\int_{\Gamma}\mathbf{1}_{\left\{
\left\vert (\bar{Y}^{\theta}-\bar{Y}^{0})(s)\right\vert \geq1\right\} }d\xi%
\right] \right) ^{1/2}ds\leq C\left( \mathbb{E}\left[ \int_{0}^{T}\int_{%
\Gamma}\left\vert (\bar{Y}^{\theta}-\bar{Y}^{0})(s)\right\vert ^{2}\,d\xi\,ds%
\right] \right) ^{1/2}\leq C\theta.
\end{align*}
Therefore,%
\begin{equation*}
\lim_{\theta\rightarrow0}\mathbb{E}\left[ \int_{0}^{T}\int_{\Gamma}\left%
\vert T^{3,\theta}(s)\right\vert \left\vert v(s)-u(s)\right\vert \left\vert
\bar {z}\right\vert \,d\xi\,ds\right] =0.
\end{equation*}

Let $\mathbf{Z\in}C([0,T];L^{2}\left( \Omega;V^{\ast}\right) )$ be defined by%
\begin{equation*}
\mathbf{Z}(t)=\int_{0}^{t}D_{\mathbf{y}}A(\mathbf{Y}^{u}(s),u(s))\mathbf{%
\tilde{Z}}(s)\,ds+\int_{0}^{t}\mathcal{G}(\bar{Y}^{u}(s),u(s),v(s)-u(s))%
\,ds,\ t\in\left[ 0,T\right] .
\end{equation*}
By the weak convergence of $\tfrac{1}{\theta}(\mathbf{Y}^{\theta}-\mathbf{Y}%
^{0})$ to $\mathbf{\tilde{Z}}$ in $L^{2}\left( \Omega\times
\lbrack0,T];V\right) $, the boundedness of $\frac{\mathbf{Y}^{\theta }(t)-%
\mathbf{Y}^{0}(t)}{\theta}$ in $L^{2}(\Omega;H)$ and the density of $\left\{
(z,\bar{z})\in V\mid z\in C_{b}^{1}(\mathcal{O})\right\} $ in $V$, we can
pass to the limit in relation (\ref{rel_weak_theta}) and obtain that, for
every $t\in\left[ 0,T\right] $, $\mathbf{Z}(t)\in L^{2}(\Omega;H)$ and $%
\frac{\mathbf{Y}^{\theta}(t)-\mathbf{Y}^{0}(t)}{\theta}$ converges weakly to
$\mathbf{Z}(t)$ in $L^{2}(\Omega;H)$. This allows the identification%
\begin{equation*}
\mathbf{Z}(t)=\mathbf{\tilde{Z}}(t),\ \text{a.e. }t\in\lbrack0,T],
\end{equation*}
from which we can infer that $\mathbf{Z}$ is a variational solution of
equation (\ref{Eq_Var}).

The uniqueness of the solution of (\ref{Eq_Var}) is obtained by applying
Theorem 4.2.4 in \cite{PreRoc05}, for instance. A consequence of the
uniqueness is that the weak convergences stated inside this argument hold
not only on a subsequence, but on a whole right neighborhood of $0$.
\end{proof}

\subsection{Necessary conditions of optimality}

In this section we will derive, in the form of a maximum principle,
necessary conditions for an admissible control to be optimal. Let us define
the \emph{Hamiltonian} $\mathcal{H}:V\times L^{2}\left( \Gamma;U\right)
\times V\times L_{2}(H)\rightarrow\mathbb{R}$ by%
\begin{equation*}
\mathcal{H}(\mathbf{y},u,\mathbf{p},q):={}_{V^{\ast}}\left\langle A(\mathbf{y%
},u),\mathbf{p}\right\rangle _{V}+L(\mathbf{y},u)+\func{tr}(qB).
\end{equation*}

\begin{theorem}
Let $u^{\ast}$ be an optimal control. Then, a.s., $d\xi\, dt$-a.e.,%
\begin{equation}
\left[ \bar{P}^{u^{\ast}}(t,\xi)D_{u}\gamma(\xi,\bar{Y}^{u^{\ast}}(t,%
\xi),u^{\ast}(t,\xi))-D_{u}\bar{\ell}(\xi,\bar{Y}^{u^{\ast}}(t,\xi
),u^{\ast}(t,\xi))\right] \cdot(v-u^{\ast}(t,\xi))\leq0,\ \forall v\in U.
\label{Rel_max_princ}
\end{equation}
\end{theorem}

\noindent\textit{Remark.} This inequality is equivalent to%
\begin{equation*}
\mathcal{H}_{u}(\mathbf{Y}^{u^{\ast}}(t),u^{\ast}(t),\mathbf{P}%
^{u^{\ast}}(t),Q^{u^{\ast}}(t);v(t)-u^{\ast}(t))\geq0,\ \mathbb{P}dt\text{%
-a.e., }\forall v\in\mathcal{U}_{u^{\ast}}^{\infty},
\end{equation*}
where $\mathcal{U}_{u^{\ast}}^{\infty}$ is the set of admissible controls $v$
such that $v-u^{\ast}\in L^{\infty}\left( \Gamma;\mathbb{R}^{m}\right) $, $%
\mathbb{P}dt$-a.e. and $\mathcal{H}_{u}(\mathbf{y},u,\mathbf{p},q;w)$
denotes the directional derivative of $\mathcal{H}$ with respect to $u$ in
the direction $w$ (which exists if $w\in L^{\infty}\left( \Gamma;\mathbb{R}%
^{m}\right) $ and $u+w\in L^{2}\left( \Gamma;U\right) $). This is known as
the \emph{local form} of the maximum principle.

\begin{proof}
As in the previous section, we will take first an arbitrary control $v$ such
that $v-u^{\ast}$ is bounded and we will use the same notations $u^{\theta}$%
, $\mathbf{Y}^{\theta}$, $\mathbf{Z}$, for $\theta\in\lbrack0,1]$. We will
also write $\mathbf{P}$, $Q$ instead of $\mathbf{P}^{u^{\ast}}$, $%
Q^{u^{\ast}}$, respectively. Let us apply It\^{o}'s formula to $\mathbf{P}%
\cdot\mathbf{Z}$:%
\begin{align*}
\left\langle \mathbf{P}(t),\mathbf{Z}(t)\right\rangle _{H}= &
-\int_{0}^{t}{}_{V^{\ast}}\left\langle (D_{\mathbf{y}}A)^{\ast}(\mathbf{Y}%
^{0}(s),u^{\ast }(s))\mathbf{P}(s),\mathbf{Z}(s)\right\rangle _{V}\,ds \\
& -\int_{0}^{t}D_{\mathbf{y}}L(\mathbf{Y}^{0}(s),u^{\ast}(s))\mathbf{Z}%
(s)\,ds \\
& +\int_{0}^{t}{}_{V^{\ast}}\left\langle D_{\mathbf{y}}A(\mathbf{Y}%
^{0}(s),u^{\ast}(s))\mathbf{Z}(s)+\mathcal{G}(\bar{Y}^{0}(s),u^{\ast
}(s),v(s)-u^{\ast}(s)),\mathbf{P}(s)\right\rangle _{V}\,ds \\
& -\int_{0}^{t}\left\langle \mathbf{Z}(s),Q(s)\,dW(s)\right\rangle _{H},\
\forall t\in\lbrack0,T].
\end{align*}
Therefore, letting $t=T$ and taking expectation, we get%
\begin{align}
\mathbb{E}\left\langle D_{\mathbf{y}}\Psi(\mathbf{Y}^{0}(T)),\mathbf{Z}%
(T)\right\rangle _{H}= & \mathbb{E}\left[ \int_{0}^{T}{}_{V^{\ast}}\left%
\langle \mathcal{G}(\bar{Y}^{0}(s),u^{\ast}(s),v(s)-u^{\ast }(s)),\mathbf{P}%
(s)\right\rangle _{V}\,ds\right]  \label{rel_duality} \\
& -\mathbb{E}\left[ \int_{0}^{T}D_{\mathbf{y}}L(\mathbf{Y}^{0}(s),u^{\ast
}(s))\mathbf{Z}(s)\,ds\right] .  \notag
\end{align}
On the other hand, since $u^{\ast}$ is an optimal control, $J(u^{\ast})\leq
J(u^{\theta})$ for any $\theta\in(0,1)$, \textit{i.e.}%
\begin{equation*}
\mathbb{E}\left[ \int_{0}^{T}\left( L(\mathbf{Y}^{\theta}(t),u^{\theta
}(t))-L(\mathbf{Y}^{0}(t),u^{\ast}(t))\right) dt+\Psi(\mathbf{Y}^{\theta
}(T))-\Psi(\mathbf{Y}^{0}(T))\right] \geq0,
\end{equation*}
which is equivalent to%
\begin{align*}
& \mathbb{E}\left[ \int_{0}^{T}\int_{0}^{1}D_{\mathbf{y}}L(\mathbf{Y}%
^{0}(t)+\lambda(\mathbf{Y}^{\theta}-\mathbf{Y}^{0})(t),u^{\lambda\theta
}(t))d\lambda\frac{(\mathbf{Y}^{\theta}-\mathbf{Y}^{0})(t)}{\theta }\,dt%
\right] \\
& +\mathbb{E}\left[ \int_{0}^{T}\int_{0}^{1}L_{u}(\mathbf{Y}^{0}(t)+\lambda(%
\mathbf{Y}^{\theta}-\mathbf{Y}^{0})(t),u^{\lambda\theta
}(t);v(t)-u^{\ast}(t))d\lambda\,dt\right] \\
& +\mathbb{E}\left[ \int_{0}^{1}D_{\mathbf{y}}\Psi(\mathbf{Y}^{0}(T)+\lambda(%
\mathbf{Y}^{\theta}-\mathbf{Y}^{0})(T))\,d\lambda\frac {(\mathbf{Y}^{\theta}-%
\mathbf{Y}^{0})(T)}{\theta}\right] \geq0.
\end{align*}
Here, $L_{u}(\mathbf{y},u;w)$ denotes the directional derivative of $L$ with
respect to $u$ in the direction $w$. Passing to the limit as $\theta
\rightarrow0$, by the weak convergence property stated in Proposition \ref%
{P_var_eq} and similar arguments as in its proof, we obtain%
\begin{equation*}
\mathbb{E}\left\{ \int_{0}^{T}\left[ D_{\mathbf{y}}L(\mathbf{Y}%
^{0}(t),u^{\ast}(t))\mathbf{Z}(t)+L_{u}(\mathbf{Y}^{0}(t),u^{\ast
}(t);v(t)-u^{\ast}(t))\right] dt\right\} \geq-\mathbb{E}\left\langle
D_{y}\Psi(\mathbf{Y}^{0}(T)),\mathbf{Z}(T)\right\rangle _{H}.
\end{equation*}
Combining this inequality with relation (\ref{rel_duality}), we derive%
\begin{equation*}
\mathbb{E}\left\{ \int_{0}^{T}\left[ _{V^{\ast}}\left\langle \mathcal{G}(%
\bar{Y}^{0}(s),u^{\ast}(s),v(s)-u^{\ast}(s)),\mathbf{P}(s)\right\rangle
_{V}+L_{u}(\mathbf{Y}^{0}(t),u^{\ast}(t);v(t)-u^{\ast}(t))\right] ds\right\}
\geq0,
\end{equation*}
\textit{i.e.}%
\begin{equation*}
\mathbb{E}\left\{ \int_{0}^{T}\int_{\Gamma}\left[ \bar{P}(t,\xi)D_{u}\gamma(%
\xi,\bar{Y}^{0}(t,\xi),u^{\ast}(t,\xi))-D_{u}\bar{\ell}(\xi,\bar{Y}%
^{0}(t,\xi),u^{\ast}(t,\xi))\right] \cdot(v(t,\xi)-u^{\ast}(t,\xi
))\,d\xi\,dt\right\} \leq0.
\end{equation*}
Since the control $v$ such that $v-u^{\ast}$ is bounded is chosen
arbitrarily, we can infer easily that a.s., $d\xi dt$-a.e.
\begin{equation*}
\left[ \bar{P}(t,\xi)D_{u}\gamma(\xi,\bar{Y}^{0}(t,\xi),u^{\ast}(t,%
\xi))-D_{u}\bar{\ell}(\xi,\bar{Y}^{0}(t,\xi),u^{\ast}(t,\xi))\right]
\cdot(v-u^{\ast}(t,\xi))\leq0,\ \forall v\in U.
\end{equation*}
\end{proof}

\subsection{Sufficient conditions of optimality}

In this section we show that condition (\ref{Rel_max_princ}) is, under some
supplementary assumptions, sufficient for the optimality of a given control.

\begin{theorem}
\label{Th: suff cond}Let $u^{\ast}$ be a control satisfying (\ref%
{Rel_max_princ}). If the mappings $\Psi$ and
\begin{equation}
\begin{aligned} & V\times L^{2}\left( \Gamma;U\right) \to \mathbb R \\
&(\mathbf{y},u)\mapsto\mathcal{H}(\mathbf{y},u,\mathbf{P}^{u^{%
\ast}}(t),Q^{u^{\ast}}(t)) \end{aligned}   \label{convex_map}
\end{equation}
are convex a.s., $dt$-a.e., then $u^{\ast}$ is optimal.
\end{theorem}

\noindent\textit{Remark.} Under the above convexity hypothesis, (\ref%
{Rel_max_princ}) becomes equivalent to%
\begin{equation*}
u^{\ast}(t)\in\func{argmin}\mathcal{H}(\mathbf{Y}^{u^{\ast}},\cdot,\mathbf{P}%
^{u^{\ast}}(t),Q^{u^{\ast}}(t)),\ \mathbb{P}dt\text{-a.e.,}
\end{equation*}
which is the \emph{global form} of the maximum principle.

\begin{proof}
For an admissible control $v$ such that $v-u^{\ast}$ is bounded, let us
apply It\^{o}'s formula to $\mathbf{P}^{u^{\ast}}\cdot\mathbf{(Y}^{v}-%
\mathbf{Y}^{u^{\ast}}\mathbf{)}$:%
\begin{align}
\left\langle \mathbf{P}^{u^{\ast}}(t),\mathbf{Y}^{v}(t)-\mathbf{Y}%
^{u^{\ast}}(t)\right\rangle _{H} & =-\int_{0}^{t}{}_{V^{\ast}}\left\langle
(D_{\mathbf{y}}A)^{\ast}(\mathbf{Y}^{u^{\ast}}(s),u^{\ast}(s))\mathbf{P}%
^{u^{\ast}}(s),\mathbf{Y}^{v}(s)-\mathbf{Y}^{u^{\ast}}(s)\right\rangle _{V}
\, ds  \label{rel_duality_dif} \\
& -\int_{0}^{t}D_{\mathbf{y}}L(\mathbf{Y}^{u^{\ast}}(s),u^{\ast }(s))(%
\mathbf{Y}^{v}(s)-\mathbf{Y}^{u^{\ast}}(s)) \, ds  \notag \\
& +\int_{0}^{t}{}_{V^{\ast}}\left\langle A(\mathbf{Y}^{v}(s),v(s))-A(\mathbf{%
Y}^{u^{\ast}}(s),u^{\ast}(s)),\mathbf{P}^{u^{\ast}}(s)\right\rangle _{V} \,
ds  \notag \\
& -\int_{0}^{t}\left\langle \mathbf{Y}^{v}(s)-\mathbf{Y}^{u^{%
\ast}}(s),Q^{u^{\ast}}(s) \, dW(s)\right\rangle _{H}.  \notag
\end{align}

Since the map $\mathcal{H}(\cdot,\cdot,\mathbf{P}^{u^{\ast}}(t),Q^{u^{%
\ast}}(t))$ is convex, we have%
\begin{multline*}
\mathcal{H}(\mathbf{Y}^{v}(t),v(t),\mathbf{P}^{u^{\ast}}(t),Q^{u^{\ast}}(t))-%
\mathcal{H}(\mathbf{Y}^{u^{\ast}}(t),u^{\ast}(t),\mathbf{P}%
^{u^{\ast}}(t),Q^{u^{\ast}}(t))\geq \\
\mathcal{H}_{(\mathbf{y},u)}\left( \mathbf{Y}^{u^{\ast}}(t),u^{\ast }(t),%
\mathbf{P}^{u^{\ast}}(t),Q^{u^{\ast}}(t);(\mathbf{Y}^{v}(t)-\mathbf{Y}%
^{u^{\ast}}(t),v(t)-u^{\ast}(t))\right) \\
={}_{V^{\ast}}\left\langle (D_{\mathbf{y}}A)^{\ast}(\mathbf{Y}%
^{u^{\ast}}(t),u^{\ast}(t))\mathbf{P}^{u^{\ast}}(t),\mathbf{Y}^{v}(t)-%
\mathbf{Y}^{u^{\ast}}(t)\right\rangle _{V}+D_{\mathbf{y}}L(\mathbf{Y}%
^{u^{\ast}}(t),u^{\ast}(t))(\mathbf{Y}^{v}(t)-\mathbf{Y}^{u^{\ast}}(t)) \\
-\int_{\Gamma}\bar{P}^{u^{\ast}}(t,\xi)D_{u}\gamma(\xi,\bar{Y}%
^{u^{\ast}}(t,\xi),u^{\ast}(t,\xi))\cdot(v(t,\xi)-u^{\ast}(t,\xi))\,d\xi \\
+L_{u}(\mathbf{Y}^{u^{\ast}}(t),u^{\ast}(t);v(t)-u^{\ast}(t)),
\end{multline*}
where $\mathcal{H}_{(\mathbf{y},u)}(\mathbf{y},u,\mathbf{p},q;\left( \mathbf{%
w},w\right) )$ denotes the directional derivative of $\mathcal{H}$ with
respect to $(\mathbf{y},u)$ in the direction $\left( \mathbf{w},w\right) $.
We make the remark that%
\begin{equation*}
\int_{\Gamma}\bar{P}^{u^{\ast}}(t,\xi)D_{u}\gamma(\xi,\bar{Y}%
^{u^{\ast}}(t,\xi),u^{\ast}(t,\xi))\cdot(v(t,\xi)-u^{\ast}(t,\xi))\,d\xi
\end{equation*}
may be infinite, but exists, by (\ref{Rel_max_princ}). From relation (\ref%
{rel_duality_dif}) we get%
\begin{align*}
\mathbb{E}\left\langle D_{\mathbf{y}}\Psi(\mathbf{Y}^{u^{\ast}}(T)),\mathbf{Y%
}^{v}(T)-\mathbf{Y}^{u^{\ast}}(T)\right\rangle _{H}\geq & \mathbb{E}\left[
\int_{0}^{T}\int_{\Gamma}\bar{P}^{u^{\ast}}(t,\xi)D_{u}\gamma(\xi,\bar {Y}%
^{u^{\ast}}(t,\xi),u^{\ast}(t,\xi))\cdot(v(t,\xi)-u^{\ast}(t,\xi ))\,d\xi\,dt%
\right] \\
& -\mathbb{E}\left[ \int_{0}^{T}\left[ L(t,\mathbf{Y}^{v}(t),v(t))-L(t,%
\mathbf{Y}^{u^{\ast}}(t),u^{\ast}(t))\right] dt\right] \\
& +\mathbb{E}\left[ \int_{0}^{T}L_{u}(\mathbf{Y}^{u^{\ast}}(t),u^{\ast
}(t);v(t)-u^{\ast}(t))\,dt\right] .
\end{align*}
The convexity of $\Psi$ implies that%
\begin{equation*}
\mathbb{E}\left[ \Psi(\mathbf{Y}^{v}(T))-\Psi(\mathbf{Y}^{u^{\ast}}(T))%
\right] \geq\mathbb{E}\left\langle D_{\mathbf{y}}\Psi(\mathbf{Y}^{u}(T)),%
\mathbf{Y}^{v}(T)-\mathbf{Y}^{u^{\ast}}(T)\right\rangle _{H};
\end{equation*}
consequently%
\begin{multline*}
J(v)-J(u^{\ast})\geq \\
\mathbb{E}\left\{ \int_{0}^{T}\int_{\Gamma}\left[ -\bar{P}%
^{u^{\ast}}(t,\xi)D_{u}\gamma(\xi,\bar{Y}^{u^{\ast}}(t,\xi),u^{\ast}(t,%
\xi))+D_{u}\bar{\ell}(\xi,\bar{Y}^{u^{\ast}}(t,\xi),u^{\ast}(t,\xi))\right]
\cdot(v(t,\xi)-u^{\ast}(t,\xi))\,d\xi\,dt\right\} .
\end{multline*}
By relation (\ref{Rel_max_princ}), the right-hand side of the above
inequality is positive, so $J(v)\geq J(u^{\ast})$.

If $v-u^{\ast}$ is not bounded, we can take, for $n\geq1$,%
\begin{equation*}
v_{n}(t,\xi):=\left\{
\begin{array}{ll}
v(t,\xi), & \left\vert v\left( t,\xi\right) -u^{\ast}(t,\xi)\right\vert \leq
n; \\
u^{\ast}(t,\xi), & \left\vert v\left( t,\xi\right) -u^{\ast}(t,\xi
)\right\vert >n.%
\end{array}
\right.
\end{equation*}
Applying It\^{o}'s formula to $\mathbf{Y}^{v_{n}}(t)-\mathbf{Y}^{v}(t)$, we
get, by the properties of $\mathbf{a}$ and $\gamma$,%
\begin{align*}
& \left\Vert \mathbf{Y}^{v_{n}}(t)-\mathbf{Y}^{v}(t)\right\Vert
_{H}^{2}+2\delta\int_{0}^{t}\left\Vert \mathbf{Y}^{v_{n}}(s)-\mathbf{Y}%
^{v}(s)\right\Vert _{V}^{2}\,ds \\
& \leq2\int_{0}^{t}\int_{\Gamma}\left[ \gamma(\xi,\bar{Y}^{v}(s),v_{n}(s))-%
\gamma(\xi,\bar{Y}^{v}(s),v(s))\right] \left( \bar{Y}^{v_{n}}(s)-\bar{Y}%
^{v}(s)\right) \,d\xi\,ds \\
& \leq C\int_{0}^{T}\int_{\Gamma}\left\vert \gamma(\xi,\bar{Y}%
^{v}(s),v_{n}(s))-\gamma(\xi,\bar{Y}^{v}(s),v(s))\right\vert ^{2}\,d\xi
\,ds+\delta\int_{0}^{T}\int_{\Gamma}\left\vert \bar{Y}^{v_{n}}(s)-\bar{Y}%
^{v}(s)\right\vert ^{2}\,d\xi\,ds.
\end{align*}
Therefore, by the dominated convergence theorem,%
\begin{equation*}
\lim_{n\rightarrow\infty}\sup_{t\in\lbrack0,T]}\left\Vert \mathbf{Y}%
^{v_{n}}(t)-\mathbf{Y}^{v}(t)\right\Vert _{H}^{2}=0.
\end{equation*}
This implies that $\lim_{n\rightarrow\infty}J(v_{n})=J(v)$; hence $J(v)\geq
J(u^{\ast})$.
\end{proof}

\noindent\textbf{Example.} The convexity hypothesis for $\mathcal{H}(\mathbf{%
\cdot},\cdot,\mathbf{P}^{u^{\ast}}(t),Q^{u^{\ast}}(t))$ is hard to verify in
practice, since the direction of $\nabla P^{u^{\ast}}$ and the sign of $\bar{%
P}^{u^{\ast}}$ are not \emph{a priori} determinable. However, under
convexity assumptions on the coefficients, we just need to strengthen
condition (\ref{Rel_max_princ}) in order to derive a sufficient optimality
condition.

We will take $\mathbf{a}(x,\zeta)=\zeta,\ (x,\zeta)\in\mathcal{O}\times%
\mathbb{R}^{n}$ (or, more general, linear with respect to $\zeta$).
Moreover, the functions $\ell(x,\cdot)$, $\psi(x,\cdot)$ and $\bar{\psi}%
(\xi,\cdot)$ are supposed to be convex, $dx$-a.e. on $\mathcal{O}$,
respectively $d\xi$-a.e. on $\Gamma$. For $\sigma\in\{-1,1\}$, on $\gamma$
and $\bar{\ell}$ we impose that:

\begin{itemize}
\item $(\bar{y},u)\mapsto-\sigma\gamma(\xi,\bar{y},u)$ is convex, $d\xi$%
-a.e. on $\Gamma$;

\item $(\bar{y},u)\mapsto\bar{\ell}(\xi,\bar{y},u)$ is convex, $d\xi$-a.e.
on $\Gamma$.
\end{itemize}

Let, for $(\xi,\bar{y},u)\in\Gamma\times\mathbb{R}\times U$,%
\begin{equation*}
S(\xi,\bar{y},u):=\left\{ \alpha\in\mathbb{R}\mid\alpha D_{u}\gamma(\xi ,%
\bar{y},u)-D_{u}\bar{\ell}(\xi,\bar{y},u)\in\mathcal{N}_{U}(u)\right\} ,
\end{equation*}
where $\mathcal{N}_{U}(u)$ is the exterior normal cone to $U$ in $u$ if $%
u\in\partial U$ and $\mathcal{N}_{U}(u)=\{0\}$ if $u\in\limfunc{int}U$.

A sufficient condition of optimality for an admissible control $u^{\ast}$ is
then%
\begin{equation}
\bar{P}^{u^{\ast}}(t,\xi)\in S(\xi,\bar{Y}^{u^{\ast}}(t,\xi),u^{\ast}(t,%
\xi))\cap\sigma\mathbb{R}_{+},\ d\xi dt\text{-a.e.}
\label{Rel_suff_max_princ}
\end{equation}
This condition is obviously equivalent to (\ref{Rel_max_princ}) when $S(\xi,%
\bar{y},u)\cap\sigma\mathbb{R}_{-}^{\ast}=\emptyset,\ \forall\bar{y}\in%
\mathbb{R},\ d\xi$-a.e.

\section{Existence of an optimal control}

Let now study the problem of the existence of an optimal control under the
convexity conditions on the coefficients of the cost functional and
linearity of control.


Assume that $U$ is bounded and:

\begin{description}
\item[(C$_{3}$)] $\gamma(\xi,\bar{y},u)=\tilde{\gamma}(\xi,\bar{y})+\beta
(\xi)\cdot u$, where $\tilde{\gamma}$ satisfies conditions (C$_{0}$)--(C$_{2}
$) and $\beta\in L^{\infty}(\Gamma;\mathbb{R}^{m})$;

\item[(F$_{2}$)] $\psi(x,\cdot)$ and $\bar{\psi}(\xi,\cdot)$ are convex, $dx$%
-a.e. on $\mathcal{O}$, respectively $d\xi$-a.e. on $\Gamma$;

\item[(L$_{2}$)] $\ell(x,\cdot)$ and $\bar{\ell}(\xi,\cdot)$ are convex, $dx$%
-a.e. on $\mathcal{O}$, respectively $d\xi$-a.e. on $\Gamma$.
\end{description}

\noindent\textit{Remark.} Notice that our assumptions, although stringent,
cover most of the cases in the literature. For instance, Debussche, Fuhrman
and Tessitore \cite{DebFuhTes07}, Fabbri and Goldys \cite{Fabbri2008} and
Bonaccorsi, Confortola and Mastrogiacomo \cite{BonConMas08} consider linear
control problems on the boundary (for Neumann, Dirichlet and dynamic
boundary conditions, respectively), and all those papers are concerned with
the one-dimensional problem. These papers, further, consider linear
quadratic term in the cost functional, that hence satisfy assumptions (F$_{2}
$) and (L$_{2}$).

On the other hand, in this paper we do not consider the structure condition
that is necessary to apply the forward-backward approach of Fuhrman and
Tessitore \cite{FuhTes02}, i.e., the condition that the control and the
noise enters the equation with the same operator in front of them.

\medskip

\begin{theorem}
\label{Th: Exist}Under the above assumptions, there exists at least an
optimal control.
\end{theorem}

The necessary condition (\ref{Rel_max_princ}) provides more information
about the optimal control whose existence is guaranteed by the above result.
In fact, it can be written as%
\begin{equation*}
\beta(\xi)\bar{P}^{u^{\ast}}(t,\xi)-D_{u}\bar{\ell}(\xi,\bar{Y}%
^{u^{\ast}}(t,\xi),u^{\ast}(t,\xi))\in\mathcal{N}_{U}(u^{\ast}(t,\xi)),\
d\xi dt\text{-a.e.}
\end{equation*}
(recall that $\mathcal{N}_{U}(u)$ is the exterior normal cone to $U$ in $u$
if $u\in\partial U$ and $\mathcal{N}_{U}(u)=\{0\}$ if $u\in\limfunc{int}U$).

\begin{proof}
By It\^{o}'s formula applied to $\left\Vert Y^{u}\right\Vert _{H}^{2}$, it
is clear that%
\begin{equation*}
\mathbb{E}\sup_{t\in\lbrack0,T]}\left\Vert Y^{u}(t)\right\Vert _{H}^{2}\leq
C,
\end{equation*}
for every control $u$ (we recall that we use the generic term $C$ for
positive constants, whose values can change from one place to another).
Since $U$ is bounded, $J$ is bounded, too. Let $\left( u_{n}\right) $ be a
sequence of controls such that $J(u_{n})\searrow\inf_{u\in\mathcal{U}}J(u)$.
There exists $u^{\ast}\in L^{2}\left( \Omega\times\lbrack0,T];L^{2}(\Gamma;%
\mathbb{R}^{m})\right) $ such that a subsequence of $\left( u_{n}\right) $
converges weakly to $u^{\ast}$. Without restricting the generality, we can
suppose that the whole sequence converges to $u^{\ast}$.

For the sake of simplicity, let us denote $\mathbf{Y}^{n}:=\mathbf{Y}^{u^{n}}
$. Let us show that $\mathbf{Y}^{n}$ converges to $\mathbf{Y}^{u^{\ast}}$.
We have that, exactly as in (\ref{rel_bound}), that
\begin{equation*}
\sup_{n\in\mathbb{N}}\left[ \sup_{t\in\lbrack0,T]}\mathbb{E}\left\Vert
\mathbf{Y}^{n}(t)\right\Vert _{H}^{2}+\mathbb{E}\int_{0}^{T}\left\Vert
\mathbf{Y}^{n}(t)\right\Vert _{V}^{2}dt\right] <+\infty.
\end{equation*}
Consequently, the sequences $\left( \mathbf{a}(\cdot,\nabla Y^{n})\right)
_{n\geq1}$ and $\left( \tilde{\gamma}(\cdot,\bar{Y}^{n})\right) _{n\geq1}$
are also bounded in $L^{2}\left( \Omega\times\lbrack0,T];L^{2}\left(
\mathcal{O}\right) \right) $, respectively in $L^{2}\left( \Omega
\times\lbrack0,T];L^{2}\left( \Gamma\right) \right) $. Therefore, at least
on a subsequence:

\begin{itemize}
\item $\mathbf{Y}^{n}$ converges weakly in $L^{2}\left( \Omega\times
\lbrack0,T];V\right) $ to a process $\mathbf{Y}^{\ast}=(Y^{\ast},\bar {Y}%
^{\ast})$;

\item $\mathbf{Y}^{n}(t)$ converges weakly in $L^{2}\left( \Omega;H\right) $
to $\mathbf{Y}^{\ast}(t)$ for every $t\in\lbrack0,T]$;

\item $\mathbf{a}(\cdot,\nabla Y^{n})$ converges weakly in $L^{2}\left(
\Omega\times\lbrack0,T];L^{2}\left( \mathcal{O}\right) \right) $ to a
process $\chi$;

\item $\tilde{\gamma}(\cdot,\bar{Y}^{n})$ converges weakly in $L^{2}\left(
\Omega\times\lbrack0,T];L^{2}\left( \Gamma\right) \right) $ to a process $%
\varkappa$.
\end{itemize}

If $\mathbf{z}=(z,\bar{z})\in V$, then%
\begin{multline*}
\left\langle \mathbf{Y}^{n}(t),\mathbf{z}\right\rangle _{H}=\left\langle
\mathbf{y}_{0},\mathbf{z}\right\rangle _{H}-\int_{0}^{t}\left[ \left\langle
\mathbf{a}(\cdot,\nabla Y^{n}(s)),\nabla z\right\rangle _{L^{2}\left(
\mathcal{O}\right) }+\left\langle \tilde{\gamma}(\cdot,\bar{Y}^{n}(s)),\bar{z%
}\right\rangle _{L^{2}(\Gamma)}\right] \,ds \\
-\int_{0}^{t}\int_{\Gamma}\beta(\xi)\bar{z}\cdot u^{n}(s,\xi)\,d\xi
\,ds+\left\langle \int_{0}^{t}B\,dW(s),\mathbf{z}\right\rangle _{H},\
t\in\lbrack0,T].
\end{multline*}
Passing to the limit in this relation, we obtain, for $t\in\lbrack0,T]$,%
\begin{equation*}
\left\langle \mathbf{Y}^{\ast}(t),\mathbf{z}\right\rangle _{H}=\left\langle
\mathbf{y}_{0},\mathbf{z}\right\rangle _{H}-\int_{0}^{t}\left[ \left\langle
\chi(s),\nabla z\right\rangle _{L^{2}\left( \mathcal{O}\right)
}+\left\langle \varkappa(s),\bar{z}\right\rangle _{L^{2}(\Gamma)}\right]
ds-\int_{0}^{t}\int_{\Gamma}\beta(\xi)\bar{z}\cdot u^{\ast}(s,\xi
)\,d\xi\,ds+\left\langle \int_{0}^{t}B\,dW(s),\mathbf{z}\right\rangle _{H},
\end{equation*}
meaning that $\mathbf{Y}^{\ast}$ satisfies the relation%
\begin{equation*}
\mathbf{Y}^{\ast}(t)=\mathbf{y}_{0}+\int_{0}^{t}\tilde{A}(s)\,ds+%
\int_{0}^{t}B\,dW(s),\ t\in\lbrack0,T],
\end{equation*}
where the $V^{\ast}$-valued, square-integrable process $\tilde{A}$ is
defined by%
\begin{equation*}
_{V^{\ast}}\left\langle \tilde{A}(s),(z,\bar{z})\right\rangle _{V}:=-\int_{%
\mathcal{O}}\chi(s)\cdot\nabla z\,dx-\int_{\Gamma}\left[ \varkappa(s)+\beta(%
\xi)u^{\ast}(s)\right] \bar{z}\,d\xi,\ (z,\bar{z})\in V.
\end{equation*}
In order to assert that $\mathbf{Y}^{\ast}=\mathbf{Y}^{u^{\ast}}$, we have
to prove the identification $\tilde{A}(s)=A(\mathbf{Y}^{\ast}(s),u^{\ast}(s))
$, $dt$-a.s. For that, we will use some results from the theory of maximal
monotone operators (see \cite{Bar10}, for example).

We have that, $\mathbb{P}$-a.s., $\mathbf{Y}^{n}(\cdot)-\mathbf{Y}%
^{1}(\cdot)\in W^{1,2}(0,T;V^{\ast})$ and
\begin{equation*}
\frac{d}{dt}\left( \mathbf{Y}^{n}(t)-\mathbf{Y}^{1}(t)\right) =A(\mathbf{Y}%
^{n}(t),u^{n}(t))-A(\mathbf{Y}^{1}(t),u^{1}(t)),\ dt\text{-a.e.}
\end{equation*}
Moreover, we have%
\begin{equation*}
\left\Vert \mathbf{Y}^{n}(t)-\mathbf{Y}^{1}(t)\right\Vert _{H}^{2}+\delta
\int_{0}^{t}\left\Vert \mathbf{Y}^{n}(s)-\mathbf{Y}^{1}(s)\right\Vert
_{V}^{2}ds\leq C\left\Vert \beta\right\Vert _{L^{\infty}(\Gamma;\mathbb{R}%
^{m})},\ \forall t\in\lbrack0,T]
\end{equation*}
and%
\begin{align*}
\int_{0}^{T}\left\Vert A(\mathbf{Y}^{n}(t),u^{n}(t))-A(\mathbf{Y}%
^{1}(t),u^{1}(t))\right\Vert _{V^{\ast}}^{2}dt & \leq C\left(
1+\int_{0}^{T}\left( \left\Vert \mathbf{Y}^{n}(t)\right\Vert
_{V}^{2}+\left\Vert \mathbf{Y}^{1}(t)\right\Vert _{V}^{2}\right) dt\right) \\
& \leq C\left( 1+\int_{0}^{T}\left( \left\Vert \mathbf{Y}^{n}(t)-\mathbf{Y}%
^{1}(t)\right\Vert _{V}^{2}+\left\Vert \mathbf{Y}^{1}\right\Vert
_{V}^{2}\right) dt\right) .
\end{align*}
Consequently, the sequence $\left( \mathbf{Y}^{n}(\cdot)-\mathbf{Y}%
^{1}(\cdot)\right) _{n\geq1}$ is bounded in $L^{2}\left( 0,T;V\right) \cap
W^{1,2}(0,T;V^{\ast})$, $\mathbb{P}$-a.s. By a well-known result of Aubin
(see, for example, Theorem 1.20 in \cite{Bar10}), since the inclusion $%
V\subseteq H$ is compact, $\left( \mathbf{Y}^{n}(\cdot)-\mathbf{Y}%
^{1}(\cdot)\right) _{n\geq1}$ is relatively compact in $L^{2}\left(
0,T;H\right) $.

As we already have that $\left( \mathbf{Y}^{n}-\mathbf{Y}^{1}\right)
_{n\geq1}$ converges weakly to $\mathbf{Y}^{\ast}-\mathbf{Y}^{1}$ in $%
L^{2}\left( \Omega\times\lbrack0,T];V\right) $, we infer that $\left(
\mathbf{Y}^{n}(\cdot)-\mathbf{Y}^{1}(\cdot)\right) _{n\geq1}$ converges
strongly to $\mathbf{Y}^{\ast}(\cdot)-\mathbf{Y}^{1}(\cdot)$ in $L^{2}\left(
0,T;H\right) $, $\mathbb{P}$-a.s. By the dominated convergence theorem $%
\mathbf{Y}^{n}$ converges strongly to $\mathbf{Y}^{\ast}$ in $L^{2}\left(
\Omega\times(0,T);H\right) $.

Let us define the operator $\mathcal{A}$ on $L^{2}\left( \Omega\times\left(
0,T\right) \times\mathcal{O}\right) \times L^{2}\left( \Omega\times\left(
0,T\right) \times\Gamma\right) $ by%
\begin{equation*}
\mathcal{A}(\zeta,\bar{y}):=\left( \mathbf{a}(\cdot,\zeta(\cdot )),\tilde{%
\gamma}(\cdot,\bar{y})\right)
\end{equation*}

Since $\mathcal{A}$ is hemicontinuous and monotone, by Theorem 2.4 in \cite%
{Bar10}, $\mathcal{A}$ is a maximal monotone operator.

It\^{o}'s formula applied to $\mathbf{Y}^{n}$, respectively $\mathbf{Y}%
^{\ast }$, yield
\begin{align*}
2\mathbb{E}\left[ \int_{0}^{T}\int_{\mathcal{O}}\mathbf{a}(x,\nabla
Y^{n}(t))\cdot\nabla Y^{n}(t)\,dx\,dt\right] & +2\mathbb{E}\left[
\int_{0}^{T}\int_{\Gamma}\tilde{\gamma}(\xi,\bar{Y}^{n}(t))\bar{Y}%
^{n}(t)\,d\xi\,dt\right] \\
=-\mathbb{E}\left\Vert \mathbf{Y}^{n}(T)\right\Vert _{H}^{2} & -2\mathbb{E}%
\left[ \int_{0}^{T}\beta(\xi)\bar{Y}^{n}(t)\cdot u^{n}(t)\,d\xi\,dt\right]
+\left\Vert \mathbf{y}_{0}\right\Vert _{H}^{2}+T\left\Vert B\right\Vert
_{L_{2}(H)}^{2}
\end{align*}
and%
\begin{align*}
2\mathbb{E}\left[ \int_{0}^{T}\int_{\mathcal{O}}\chi(t)\cdot\nabla Y^{\ast
}(t)\,dx\,dt\right] & +2\mathbb{E}\left[ \int_{0}^{T}\int_{\Gamma
}\varkappa(t)\bar{Y}^{\ast}(t)\,d\xi\,dt\right] \\
=-\mathbb{E}\left\Vert \mathbf{Y}^{\ast}(T)\right\Vert _{H}^{2} & -2\mathbb{E%
}\left[ \int_{0}^{T}\beta(\xi)\bar{Y}^{\ast}(t)\cdot u^{\ast }(t)\,d\xi\,dt%
\right] +\left\Vert \mathbf{y}_{0}\right\Vert _{H}^{2}+T\left\Vert
B\right\Vert _{L_{2}(H)}^{2}.
\end{align*}
Consequently, since the norm in $H$ is lower-semicontinuous with respect to
the weak topology, $u^{n}$ converges weakly to $u^{\ast}$ in $L^{2}\left(
\Omega\times\lbrack0,T];L^{2}(\Gamma;\mathbb{R}^{m})\right) $ and $\bar {Y}%
^{n}$ converges strongly to $\mathbf{Y}^{\ast}$ in $L^{2}\left(
\Omega\times\lbrack0,T];L^{2}(\Gamma)\right) $, we get
\begin{equation*}
\limsup_{\varepsilon\rightarrow0}\left\langle \mathcal{A}(\nabla Y^{n},\bar {%
Y}^{n}),(\nabla Y^{n},\bar{Y}^{n})\right\rangle \leq\left\langle
(\chi,\varkappa),(\nabla Y^{\ast},\bar{Y}^{\ast})\right\rangle ,
\end{equation*}
where $\left\langle \cdot,\cdot\right\rangle $ denotes the scalar product in
$L^{2}\left( \Omega\times\left( 0,T\right) \times\mathcal{O}\right) \times
L^{2}\left( \Omega\times\left( 0,T\right) \times\Gamma\right) $. By
Corollary 2.4 in \cite{Bar10},%
\begin{equation*}
(\chi,\varkappa)=\mathcal{A}(\nabla Y^{\ast},\bar{Y}^{\ast}),
\end{equation*}
\textit{i.e.}%
\begin{align*}
\chi(t) & =\mathbf{a}(\cdot,\nabla Y^{\ast}(t)),\ \mathbb{P}\text{-a.s.}%
\times dx\,dt\text{-a.e.;} \\
\varkappa(t) & =\tilde{\gamma}(\cdot,\bar{Y}^{\ast}(t)),\ \mathbb{P}\text{%
-a.s.}\times d\xi\,dt\text{-a.e.}
\end{align*}
By the uniqueness of the solution of equation (\ref{Eq_state}), we have that
$\mathbf{Y}^{\ast}=\mathbf{Y}^{u^{\ast}}$.

The functional $\mathcal{J}:L^{2}\left( \Omega\times\lbrack0,T];H\right)
\times L^{2}\left( \Omega\times\lbrack0,T];L^{2}(\Gamma;\mathbb{R}%
^{m})\right) \rightarrow\mathbb{R}$, defined by%
\begin{equation*}
\mathcal{J}(\mathbf{Y},u):=\mathbb{E}\left[ \int_{0}^{T}L(\mathbf{Y}%
(t),u(t))dt+\Psi(\mathbf{Y}(T))\right]
\end{equation*}
is strongly continuous. By conditions (F$_{2}$) and (L$_{2}$), it is also
convex and therefore weakly lower semi-continuous. As a consequence, $%
\liminf_{n\rightarrow\infty}\mathcal{J}(\mathbf{Y}^{n},u^{n})\geq \mathcal{J}%
(\mathbf{Y}^{\ast},u^{\ast})$. Since $\mathcal{J}(\mathbf{Y}%
^{n},u^{n})=J(u^{n})$, $\mathcal{J}(\mathbf{Y}^{\ast},u^{\ast})=J(u^{\ast})$
and $J(u^{n})\rightarrow\inf_{u\in\mathcal{U}}J(u)$, $u^{\ast}$ has to be an
optimal control.
\end{proof}

\noindent\textbf{Acknowledgement.} The research of A.~Z\u{a}linescu was
supported by the 2010 PRIN project: \textquotedblleft Equazioni di
evoluzione stocastiche con controllo e rumore al bordo\textquotedblright.

\end{document}